\documentclass[12pt]{amsart}

\usepackage{euler, epic,eepic,latexsym, amssymb, amscd, amsfonts, xypic, url, color, epsfig}

\input xy
\xyoption{all}

 \newlength{\baseunit}               
 \newcount{\numlines}                
\theoremstyle{plain}

\newtheorem*{tmnl}{Main Theorem}

\newtheorem{tm}{Theorem}

\newtheorem{lm}[tm]{Lemma}

\theoremstyle{definition}
\newtheorem{df}[tm]{Definition}

\newtheorem{property}{Property}

\theoremstyle{remark}
\newtheorem{rmk}[tm]{Remark}

\newcommand{\bbA}{\mathbf{A}}

\newcommand{\bbP}{\mathbf{P}}

\newcommand{\calO}{{ \mathcal O}}

\newcommand{\Spec}{ {\operatorname{Spec}\;}}

\begin{document}

\pagestyle{plain}
\title{The compactified jacobian can be nonreduced}

\author{Jesse Leo Kass}
\address[Current]{Dept.~of Mathematics, University of South Carolina, 1523 Greene Street, Columbia~SC 29208}
\address[Former]{Institut f\"{u}r algebraische Geometrie, Leibniz Universit\"{a}t Hannover, Welfengarten 1, 30060 Hannover, Germany}
\email{kassj@math.sc.edu}

\date{\today}
\begin{abstract}
We prove by explicit example that the compactified  jacobian can be nonreduced.  The example is a rational space curve of arithmetic genus $4$.  This answers a question posed by Cyril D'Souza in 1979.
\end{abstract}
\maketitle

{\parskip=12pt 

We prove that the  compactified jacobian, or moduli space of fixed degree rank $1$, torsion-free sheaves, can be nonreduced.   D'Souza and Altman--Kleiman independently proved that the compactified jacobian $\overline{J}(X)$ of a reduced and irreducible curve $X$ exists as a projective scheme \cite[Theorem~II.4.1]{dsouza79}, \cite[Theorem~8.1]{altman80}. On \cite[page~423]{dsouza79}, D'Souza asked a natural follow-up question: Does $\overline{J}(X)$ have good properties?  Is $\overline{J}(X)$ irreducible?
Is $\overline{J}(X)$ reduced?    

D'Souza wrote his paper in 1974, but it was not published until 1979.  In the intervening 5 years, his irreducibility question was answered. Altman--Iarrobino--Kleiman and  Rego independently proved that $\overline{J}(X)$ is irreducible when the singularities of $X$ are planar:
\begin{description}
	\item[(1)]   If the singularities of $X$ are planar, then $\overline{J}(X)$ is a reduced and irreducible variety with local complete intersection singularities (\cite[Theorem~(9)]{altman77}; see also \cite[Theorem~A]{rego} for irreducibility and \cite[Proposition~1.4]{granger} for an analogous result for the Hilbert scheme).

\end{description}
By contrast, when $X$ has non-planar singularities, Kleiman--Kleppe and Rego independently showed that $\overline{J}(X)$ is reducible:

\begin{description}
	\item[(2)] if $X$ has a non-planar singularity, then $\overline{J}(X)$ is reducible (\cite[Theorem~A]{rego}, \cite[Theorem~(1)]{Kleppe81}; see also \cite[Theorem~A]{kass12} for $X$ non-Gorenstein).
\end{description}
Altman--Iarrobino--Kleiman also proved that $\overline{J}(X)$ is connected \cite[Proposition~(11)]{altman77}.

Results (1) and (2) completely answer D'Souza's question about when $\overline{J}(X)$ is irreducible.  Here we answer D'Souza's question about when $\overline{J}(X)$ is reduced.

\begin{tmnl} \label{Theorem: MainTheorem}
	The compactified jacobian $\overline{J}(X)$ can be nonreduced.  
\end{tmnl}

This statement is proven as Theorem~\ref{Thm: MainTheorem} below.  That last theorem in fact establishes that $\overline{J}(X)$ can be arbitrarily nonreduced in the following sense.  Given an integer $b_0>1$, we construct a curve $X$ such that there is a rational function $f$ on $\overline{J}(X)$ satisfying 
\begin{equation}  \label{Eqn: SpecialProperty}
	f^{b_0+1}=0 \text{ but } f^{b_0} \ne 0.
\end{equation}
When $b_0=1$, $X$ can be taken to be a rational space curve of  arithmetic genus $4$.

The idea behind the construction of  $X$ is the following:  An argument with the Abel map shows that $\overline{J}(X)$ is nonreduced provided the first Quot scheme $X^{[1]}$ is nonreduced.  (The Quot scheme $X^{[1]}$  parameterizes length $1$ quotients of the dualizing sheaf $\omega$.)  When $X$ is Gorenstein, $X^{[1]}$ is isomorphic to $X$ (and hence reduced).  The geometry of $X^{[1]}$ is more interesting when $X$ is non-Gorenstein, especially when $X$ is a non-Gorenstein space curve of local Cohen--Macaulay type $2$.  By  the Hilbert--Burch theorem, such an $X$ is locally the zero locus of the minors  of a $3$-by-$2$ matrix $A \in \operatorname{Mat}_{3,2}(k[x,y,z])$, and a computation shows that $X^{[1]}$ is locally the complete intersection defined by the rows of $A$.  The curve in the Main Theorem is constructed by choosing an $A$ (the matrix in Equation~\eqref{Eqn: HBMatrix} below) so that the rows define a nonreduced curve.  Nonreducedness is not automatic: For example, when $X$ is a rational curve with a unique singularity analytically equivalent to the axes in $3$-space (the singularity $A_1 \vee L$ in the notation of \cite[Table~4]{fruhbiskruger}), a computation shows that $X^{[1]}$ is a nodal curve (and more generally if the singularities of $X$ are of finite representation type in the sense of \cite[Section~4]{kass12}, then $X^{[1]}$ is reduced).  Nevertheless, the author expects that further examples of matrices defining nonreduced curves can be constructed without difficulty.  The  matrix used in this paper was chosen so that the nonreduced structure of $X^{[1]}$ is particularly transparent.

The compactified jacobian is an example of a moduli space of stable sheaves $\operatorname{M}(V)$ on a reduced and irreducible variety $V$.  How do the results of this paper compare with past work on $\operatorname{M}(V)$?  The  author is aware of three different bodies of work.  To the author's knowledge, Igusa constructed the first example of a nonreduced moduli space of stable sheaves in \cite{igusa}.  Igusa constructed a smooth projective surface $V$ in characteristic $2$, the quotient of an abelian surface by $\mathbb{Z}/2$, with the property that the moduli space of line bundles $\operatorname{M}(V)$, a moduli space of stable sheaves, is nonreduced.  Additional  smooth projective varieties with the property that the moduli space $\operatorname{M}(V)$ of line bundles is nonreduced have been constructed by  Serre \cite[Proposition~15]{serre58}, Bombieri--Mumford \cite[Theorem~2]{bombieri}, Raynaud \cite[4.2.3]{raynaud}, Suh \cite[Theorem~1.2.1, Theorem~3.4]{suh},  and Liedtke \cite{liedtke}. Their constructions are all constructions in positive characteristic.  Similar examples cannot exist in characteristic zero because in characteristic zero a moduli space of line bundles is reduced.

In characteristic zero the first examples of a nonreduced $\operatorname{M}(V)$ are moduli spaces of stable rank $2$ vector bundles  on a smooth projective surface.  These moduli spaces and their nonreduced  structures have been intensively studied by mathematicians working on differentiable $4$-manifolds because of a relation to the Donaldson invariant. In particular, Okonek and Van de Ven have proven that, for a suitable choice of polarization, the moduli space $\operatorname{M}(V)$ of stable rank $2$ vector bundles with Chern class $(c_1,c_2)=(0,1)$ is nonreduced when $V$ is a non-singular minimal elliptic surface of geometric genus $p_{a}=0$ that contains two multiple fibers with multiplicities $3$ and $5$ respectively  \cite[Theorem~2.1 and Page~14, Remark]{okonek}.  Other results along these lines are given by Friedman--Morgan \cite[Lemma~4.5, Theorem~4.6]{friedman}, Bauer \cite[(3.1)~Theorem]{bauer92}, Kotschick \cite[Corollary~(2.9), Corollary~(4.17), Proposition~(5.8)]{kotschick}, and Bauer \cite[Section~V. Multiplicities]{bauer94}.

Vakil has proven stronger results about the nonreducedness of  $\operatorname{M}(V)$ when $V=\bbP^{4}$ is projective $4$-space.  He proved that the moduli space of stable sheaves $\operatorname{M}(\bbP^{4})$ can be arbitrarily nonreduced (in fact arbitrarily singular) in a sense that he makes precise  \cite[1.1.~Main Theorem]{vakil06}.  Payne has extended Vakil's result to a certain moduli space of toric vector bundles  \cite[Theorem~4.1] {payne}, although that moduli space is not a moduli space of stable sheaves.

\subsubsection*{Conventions}
$k$ is an algebraically closed field.  A \emph{curve} is a reduced and irreducible proper $k$-scheme of dimension $1$.  We write $\omega$ for the \emph{dualizing sheaf} of $X$.  For a  $k$-scheme $T$, we write $\omega_{T}$ for the pullback of $\omega$ by the projection $X \times_{k} T \to X$. The degree $d$ \emph{compactified jacobian} $\overline{J}^{d}(X)$ is the moduli space of rank $1$, torsion-free sheaves of degree $d$.  We abuse notation and write $\overline{J}(X)$ for $\overline{J}^{d}(X)$ when the degree $d$ is clear from context.  The degree $d$ \emph{Quot scheme} $X^{[d]}:=\operatorname{Quot}_{\omega}^{d}(X)$ is the parameter space of surjections from the dualizing sheaf $\omega$ to an $\calO_{X}$-module of length $d$.  The \emph{Abel map} is the morphism $X^{[d]} \to \overline{J}^{2g-2-d}(X)$ defined by sending a surjection to its kernel.  For the precise definitions, see  \cite[Definition~5.11, Theorem~8.1]{altman80} for $\overline{J}(X)$,  \cite[Definition~2.5, Theorem~2.6]{altman80} for $X^{[d]}$, and \cite[(5.16), (8.2)]{altman80} for the Abel map.

\section{The Example}
We prove Theorem~A by explicitly constructing a curve $X$ with the property that the projectivization $\bbP \omega$ of the dualizing sheaf is nonreduced and then deducing the nonreducedness of $\overline{J}(X)$ using an Abel map.  In order to prove that $\overline{J}(X)$ is not only nonreduced but in fact admits a function satisfying Equation~\eqref{Eqn: SpecialProperty}, we also prove a technical lemma about the behavior of nilpotence.  The reader interested only in proving  that $\overline{J}(X)$ is nonreduced, and not the stronger result that Equation~\eqref{Eqn: SpecialProperty} holds, can ignore Property~\ref{Eqn: SpecialPropertyB} and  Lemmas~\ref{Lemma: NonVanishing}, \ref{Lemma: RightOrder} below and replace the citation of Lemma~\ref{Lemma: RightOrder} in the proof of Theorem~\ref{Thm: MainTheorem} with a citation of the fact that noreducedness descends down smooth morphisms.

We now fix an integer $b_0 > 0$ and construct $X$.

\begin{df}
Define $X \subset \bbP^3_{k}$ to be the image of the morphism $\bbP^1_{k} \to \bbP^3_{k}$ defined in projective coordinates by 
\begin{equation} \label{Eqn: Parameterize}
	[S,T] \mapsto [S^{3 b_{0}+5}, S^{3 b_{0}+2} T^{3}, T^{3b_{0}+5}, S T^{3 b_{0}+4}].
\end{equation}
\end{df}
This definition is chosen so that the structure sheaf $\calO_{X}$ has a particular free resolution, the resolution \eqref{Eqn: Resolution} appearing below in the proof of Lemma~\ref{Lemma: MainLemma}.  

\begin{rmk}
The scheme $X$ is a curve of arithmetic genus $2 b_0+2$.  Indeed, by construction the normalization $\widetilde{X}$ is a rational curve, and the quotient $\calO_{\widetilde{X}}/\calO_{X}$ is supported at the singularity  $x_0 := [1,0,0,0]$ and has length $\delta(x_0)=2 b_0+2$.  In particular, when $b_0=1$, $X$ has arithmetic genus $4$.
\end{rmk}

We will show that the compactified jacobian of $X$ is nonreduced and in fact satisfies a slightly stronger condition that we now introduce.
\begin{property} \label{Eqn: SpecialPropertyB}
	We say that a $k$-scheme $V$ satisfies \textbf{Property~\ref{Eqn: SpecialPropertyB}} if there exists a nonempty open subset $W \subset V$ and a regular function $f \in H^{0}(W, \calO_{V})$ such that $f^{b_0+1}=0$ but the support of $f^{b_0}$ is $W$ (so in particular $f^{b_0} \ne 0$).
\end{property}

\begin{lm} \label{Lemma: MainLemma}
	The projectivization $ \bbP \omega := \operatorname{Proj} \calO_{X}[\omega]$ of the dualizing sheaf $\omega$ satisfies Property~\ref{Eqn: SpecialPropertyB}. Furthermore, the nonreduced open subscheme $W \subset \bbP \omega$ can be taken to be contained in the fiber $\beta^{-1}(x_0)$ of the structure morphism $\beta \colon \bbP \omega \to X$ over the singularity $x_0 \in X$.
\end{lm}
\begin{proof}
We prove the lemma by using a free resolution of $\calO_{X}$ to  compute $\bbP \omega$.  Let $X_1 \subset X$ be the open affine that is complement of the hyperplane $\{ [W,X,Y,Z] \colon W=0 \} \subset \bbP^{3}_{k}$, so $X_1 = \Spec(k[t^{3}, t^{3 b_0 + 5}, t^{3 b_0+4}])$.  The author claims $X_1$ is isomorphic to the subscheme of $\bbA^{3}_{k}$ defined by the maximal minors $m_1$, $m_2$, and $m_3$ of the matrix 
\begin{equation} \label{Eqn: HBMatrix}
	A := \begin{pmatrix}
			z		&	x^{b_0+1} 	\\
			y			&	z			\\
			x^{b_0+2}	&	y
		\end{pmatrix} \in \operatorname{Mat}_{3, 2}(k[x, y, z]).
\end{equation}
To verify the claim, observe that the homomorphism 
\begin{equation} \label{Eqn: Normalize}
	k[x,y,z]/(m_1, m_2, m_3) \to k[t^{3},  t^{3 b_0 +5}, t^{3 b_0 + 4}]
\end{equation}
	 defined by $x \mapsto t^3$, $y \mapsto t^{3 b_0 + 5}$, $z \mapsto t^{3 b_0 + 4}$ is a surjection that preserves a natural grading.  A computation shows that the graded pieces of both the source and the target of the homomorphism are $1$-dimensional, so surjectivity implies injectivity, and we have verified the claim.

 One consequence of the claim is that $\calO_{\bbA_{k}^3}/(m_1, m_2, m_3)$ is a Cohen--Macaulay curve, so if $B := \begin{pmatrix} m_1 & -m_2 & m_3 \end{pmatrix} \in \operatorname{Mat}_{1,3}(k[x,y,z])$, then  the sequence 
\begin{equation} \label{Eqn: Resolution}
	0 \to \begin{matrix} \calO_{\bbA_{k}^{3}} \\ \oplus \\ \calO_{\bbA_{k}^{3}} \end{matrix} \stackrel{A}{\longrightarrow} \begin{matrix} \calO_{\bbA_{k}^{3}} \\ \oplus \\ \calO_{\bbA_{k}^{3}} \\ \oplus \\ \calO_{\bbA_{k}^{3}}  \end{matrix} \stackrel{B}{\longrightarrow} \calO_{\bbA_{k}^{3}} \to \calO_{X_1} \to 0
\end{equation}
is an exact sequence by the Hilbert--Burch theorem \cite[Theorem~3.2]{eisenbud05}.

Using this sequence to compute the dualizing module $\omega_{X_1} = \operatorname{Ext}^{2}(\calO_{X_1}, k[x,y,z])$, we see that $\omega_{X_1}$ admits a presentation with two generators and three relations, the relations being described by the rows of $A$.  If $\beta \colon \bbP \omega \to X$ is the structure morphism, then by construction $\beta^{-1}(X_1)$ contains an open subscheme of the form  $\Spec(\calO_{X_1}[v]/(z + v x^{b_0+1}, y + v z, x^{b_0+2} +v y))$.  Elementary algebra shows 
\begin{equation} \label{Eqn: GoodOpen}
	\calO_{X_1}[v]/(z  + v x^{b_0+1}, y + v z, x^{b_0+2}+v y) = k[x,v]/( x^{b_0+1}(x +v^{3})),
\end{equation}
and so 
\begin{align*}
	W = & \Spec k[x,v, 1/(x +v^{3})]/( x^{b_0+1}(x +v^{3})), \\
	f=& x
\end{align*}
	 satisfy the desired condition.
\end{proof}

The following lemmas  allow us to control the behavior of nilpotence under the Abel map.

\begin{lm} \label{Lemma: NonVanishing}
	If $R$ is a $k$-algebra and $f \in R[t_1, \dots, t_d]$ a nonzero polynomial, then $f(\underline{r}) \in R$ is nonzero for some $\underline{r}=(r_1, \dots, r_d) \in R \times \dots \times R$.
\end{lm}
\begin{proof}
	When $R=k$, the lemma follows e.g.~from the Hilbert Nullstellensatz.  In general, pick a (possibly infinite) basis $\{ e_i \}_{i \in I}$ for $R$ as a $k$-module and write $f = \sum f_i e_i$ with $f_i \in k[t_1, \dots, t_d]$.  Some $f_i$ is nonzero by assumption.  By the case $R=k$, we have $f_i(\underline{r}) \ne 0$ for some $\underline{r}=(r_1, \dots, r_d) \in k \times \dots k$, and this tuple, considered as an element of $R \times \dots \times R$, satisfies the desired condition.
\end{proof}

\begin{lm} \label{Lemma: RightOrder}
	Let $V$ be a finite type $k$-scheme.  If $V \times_{k} \bbA_{k}^{d}$ satisfies Property~\ref{Eqn: SpecialPropertyB}, then so does $V$.
\end{lm}
\begin{proof}
	Let $W \subset V \times_{k} \bbA_{k}^{d}$ and $f$ be as in Property~\ref{Eqn: SpecialPropertyB}.  To begin, observe that if $\Spec(R) \subset V$ is an open affine, then $\Spec(R) \times_{k} \bbA_{k}^{d} \subset V \times_{k} \bbA_{k}^{d}$ satisfies Property~\ref{Eqn: SpecialPropertyB}  provided $W \cap \Spec(R) \times_{k} \bbA_{k}^{d}$ is nonempty.  In particular, we can assume that $V=\Spec(R)$ is affine, irreducible, and has no embedded components (i.e.~the zero ideal is primary).  (To see this, observe that the image of $W$ under the projection $V \times_{k} \bbA_{k}^{d} \to V$ is open, so if we choose an irreducible component that meets the image, then we can take $\Spec(R)$ to be a nonempty open in the complement of the union of all embedded components and all irreducible components except  the chosen one.)  Furthermore,  we can assume that $W$ is a basic affine open $W = \Spec R[t_1,\dots, t_d, 1/g] 
	\subset V \times_{k} \bbA_{k}^{d}$ with $g \in R[t_1, \dots, t_d]$, and then by clearing denominators we can assume $f$ is the restriction of an element  in $R[t_1, \dots, t_d]$ that we will also denote by $f$. 
	
	We now construct a tuple $\underline{r}=(r_1, \dots, r_d)$ with the property that  $f^{b_0}(\underline{r})$ has nonzero image  in  $R[1/g(\underline{r})]$. This will prove the lemma.  Pick a large integer $k_0$ so that the  nilradical $\operatorname{Nil} \subset R$ satisfies $\operatorname{Nil}^{k_0}=0$.  Since $f^{b_0}$ has nonzero image in $R[t_1, \dots, t_d, 1/g]$, we must have $g^{k_0} f^{b_0} \ne 0$ in $R[t_1, \dots, t_d]$, so by Lemma~\ref{Lemma: NonVanishing}, there exists $\underline{r} \in R \times \dots \times R$ such that $g^{k_0}(\underline{r}) \cdot f(\underline{r})^{b_0} \ne 0$.  Now consider the function  $f(\underline{r})/1$ on the basic affine open $\Spec R[ 1/g(\underline{r})] \subset \Spec R$.  We certainly have $f^{b_0+1}(\underline{r})/1=0$, and the author claims that we also have $f^{b_0}(\underline{r})/1 \ne 0$.  If the claim failed, then we would have 
$$
	g^{m}(\underline{r}) \cdot f^{b_0}(\underline{r})=0 \text{ in $R$}
$$
for some $m$, but the zero ideal of $R$ is primary (by our preliminary reduction), and $f^{b_0}(\underline{r}) \ne 0$, so we must have that $g^{m}(\underline{r})$ and hence $g(\underline{r})$ is nilpotent.  In particular, $g^{k_0}(\underline{r}) = 0$, but we chose $\underline{r}$ so that this condition does not hold.  A contradiction!  This establishes the claim, and since $f^{b_0}(\underline{r})/1 \in R[1/g(\underline{r})]$ is nonzero, its support  must equal $\Spec(R[1/g(\underline{r})])$, so $f(\underline{r})$ satisfies the desired conditions.
\end{proof}

We now prove the main theorem.
\begin{tm} \label{Thm: MainTheorem}
	$\overline{J}(X)$ satisfies Property~\ref{Eqn: SpecialPropertyB}.
\end{tm}
\begin{proof}
We prove the theorem using Lemma~\ref{Lemma: MainLemma} together with  Lemma~\ref{Lemma: RightOrder} and the Abel map.  We begin by showing that $X^{[d]}$  satisfies Property~\ref{Eqn: SpecialPropertyB} for all $d \ge 1$.  For $d=1$, this is  Lemma~\ref{Lemma: MainLemma} as $X^{[1]} = \bbP \omega$.  The identification $X^{[1]} = \bbP \omega$ is obtained by using adjunction to identify the relevant functors.  More formally, let $\beta \colon \bbP \omega \to X$ be the structure morphism, $Q_{\text{taut}}$ the tautological line bundle (that restricts to $\calO(1)$ on every fiber of $\beta$), $q_{\text{taut}} \colon \beta^{*} \omega \to Q_{\text{taut}}$ the tautological surjection, and $p_1 \colon X \times_{k} \bbP \omega \to X$ the projection morphism.  The identity homomorphism  $(\beta \times 1)^{* } p_{1}^{*}\omega = \beta^{*} \omega \to \beta^{*} \omega$ is adjoint to a homomorphism $p_{1}^{*}\omega \to (\beta \times 1)_{*} \beta^{*} \omega$, and the composition of this adjoint with $(\beta \times 1)_{*} q_{\text{taut}}$ defines a morphism $\bbP \omega \to X^{[1]}$ that is the desired isomorphism.  To see this morphism is an isomorphism, observe that the tautological quotient on $X \times_{k} X^{[1]}$ must be of the form $(\operatorname{HC} \times 1)_{*}Q$ for $Q$ some coherent module on $X^{[1]}$ and $\operatorname{HC} \colon X^{[1]} \to X$ the analogue of the Hilbert--Chow morphism, and the associated adjoint homomorphism defines the inverse morphism $X^{[1]} \to \bbP \omega$.  Under the identification $\bbP \omega \cong X^{[1]}$, the open subset $W$  from Lemma~\ref{Lemma: MainLemma} is identified with an open subset contained in the locus of quotients supported on the singularity $x_0$.

When $d >1$, $X^{[d]}$ contains an open subscheme isomorphic to $V \times_{k} W$ for $W \subset X^{[1]}$ the open subset   from Lemma~\ref{Lemma: MainLemma}  and  $V \subset X^{[d-1]}$ the open locus parameterizing quotients supported on the smooth locus $X-\{x_0\}$, and $V \times_{k} W$  is nonreduced.  In fact, the pullback of the function $f \in H^{0}(W, \calO_{X^{[1]}})$ under projection satisfies Property~\ref{Eqn: SpecialPropertyB}.  To see all this, consider the morphism $V \times_{k} W \to X^{[d]}$ that is defined by setting, for a given $k$-scheme $T$,
\begin{equation} \label{Eqn: FunctorOfPoints}
	V(T) \times W(T) \to X^{[d]}(T)
\end{equation}
equal to the function that sends a pair $(q_1, q_2)$ to the product homomorphism $q := q_1 \times q_2 \colon \omega_{T} \to Q := Q_1 \times Q_{2}$.  We prove that $V\times_{k} W \to X^{[d]}$ is an open immersion by showing that it is universally injective and formally \'{e}tale.  The function \eqref{Eqn: FunctorOfPoints} is  injective because given $q$ we can recover $q_2$ as the composition of $q$ with the localization homomorphism $Q \to Q \otimes \calO_{X_{T}, (x_0)_{T}}$ (as $Q_{2}$ is supported at $(x_0)_{T}$)  and similarly with $q_1$.  In particular, $V \times_{k} W \to X^{[d]}$ is universally injective and formally unramified.  The morphism is also formally smooth: given a closed immersion of affine schemes $T \to \widetilde{T}$ that is defined by a nilpotent ideal, an element $(q_1, q_2) \in V(T) \times W(T)$ with image $q \in X^{[d]}(T)$, and an element $\widetilde{q} \in X^{[d]}(\widetilde{T})$  lifting $q$, consider the homomorphism $\widetilde{q}_2$ that is the composition of $\widetilde{q}$ with the localization homomorphism $\widetilde{Q} \to \widetilde{Q} \otimes \calO_{X_{\widetilde{T}}, (x_0)_{\widetilde{T}}}$.  The homomorphism $\widetilde{q}_2$ is an element of $X^{[1]}(\widetilde{T})$ because its image is $\widetilde{T}$-flat and has length $1$ fibers ($Q_{2}$ and $\widetilde{Q}_{2}$ have the same fibers because  $T \to \widetilde{T}$ is bijective).  Because $\widetilde{q}_2$ lifts $q_2$, $\widetilde{q}_2 \in W(\widetilde{T})$ (as the morphism $W \to X^{[1]}$ is an open immersion).  Defining $\widetilde{q}_1 \in V(\widetilde{T})$ analogously, we have constructed  a pair $(\widetilde{q}_1, \widetilde{q}_2) \in W(\widetilde{T}) \times V(\widetilde{T})$ that lifts $(q_1, q_2)$ and maps to $\widetilde{q}$, proving that $V \times_{k} W \to X^{[d]}$ is formally smooth and thus is an open immersion.  The conditions of Property~\ref{Eqn: SpecialPropertyB} are satisfied by the pullback $f \otimes 1 \in H^{0}(V \times_{k} W, \calO_{V \times_{k} W})$ of the function from Lemma~\ref{Lemma: MainLemma} under the projection $V \times_{k} W \to W$ (since $(f \otimes 1)^{b_0} \ne 0$ by the faithful flatness of $V \times_{k} W \to W$).

The proof is now complete.  Once $d$ is sufficiently large, the Abel map $X^{[d]} \to \overline{J}(X)$ is a (Zariski-locally trivial) $\bbP_{k}^{d-g}$-bundle \cite[Theorem~(8.4(v))]{altman80}.  In particular, we can find covers by open subsets $U \subset \overline{J}^{d}(X)$ and $U \times \bbA_{k}^{d-g} \cong V \subset X^{[d]}$ with the property that the restriction of the Abel map to $V$ is the projection $V=U \times \bbA_{k}^{d-g} \to U$, and so the compactified jacobian satisfies Property~\ref{Eqn: SpecialPropertyB} by Lemma~\ref{Lemma: RightOrder}.
\end{proof}

\begin{rmk}
	We can make Theorem~\ref{Thm: MainTheorem} more explicit. The proof shows that $\overline{J}(X)$ is nonreduced at the kernel of a surjection $\omega \to k(x_0)$ from the dualizing sheaf to the skyscraper sheaf $k(x_0)$ supported at the singularity $x_0 \in X$.  For example, define $\eta_1$ and $\eta_2$ to be the two generators of $\omega_{X_1}$ from Lemma~\ref{Lemma: MainLemma}.  Then $\overline{J}(X)$ is nonreduced at the submodule of $\omega$ generated by $\eta_1, x \eta_2$ on $X_1$ and equal to $\omega$ away from $x_0 \in  X$.
\end{rmk}

\section*{Acknowledgments}
The author would like to thank the anonymous referee, Daniel Erman and Ravi Vakil for helpful feedback, Klaus Hulek and Junecue Suh for answering questions about the literature on moduli spaces, and  Louisa McClintock for helpful comments on exposition. 

This work was completed while the author was a Wissenschaftlicher Mitarbeiter at the Institut f\"{u}r Algebraische Geometrie, Leibniz Universit\"{a}t Hannover. During that time, the author was supported by an AMS-Simons Travel Grant.
}

\providecommand{\bysame}{\leavevmode\hbox to3em{\hrulefill}\thinspace}
\providecommand{\MR}{\relax\ifhmode\unskip\space\fi MR }

\providecommand{\MRhref}[2]{%
  \href{http://www.ams.org/mathscinet-getitem?mr=#1}{#2}
}
\providecommand{\href}[2]{#2}

\end{document}